\documentclass[reqno, 11pt]{amsart}
\usepackage[hmargin=3cm,vmargin=3cm]{geometry}
\usepackage{amsmath,amssymb,amsthm,}
\usepackage{verbatim,epsfig,graphics,hyperref}
\usepackage{marginnote}
\usepackage{color}
\usepackage{mathrsfs}
\usepackage{tikz}
\usepackage{svg}
\usepackage{esint}

\newtheorem{theorem}{Theorem}[section]
\newtheorem{proposition}[theorem]{Proposition}

\newtheorem{lemma}[theorem]{Lemma}

\theoremstyle{definition}

\newtheorem{remark}[theorem]{Remark}

\newcommand{\del}{\partial}
\newcommand{\eps}{{\varepsilon_\circ}}
\newcommand{\dd}{\, \mathrm{d}}

\DeclareMathOperator{\tube}{w}

\DeclareMathOperator{\DtN}{\mathscr{D}}
\DeclareMathOperator{\inj}{inj}
\DeclareMathOperator{\dist}{dist}

\DeclareMathOperator{\cusp}{{\mathscr{K}}}
\newcommand{\thick}{{\scriptscriptstyle{\text{thick}}}}
\newcommand{\thin}{\scriptscriptstyle{\text{thin}}}

\newcommand{\double}{\mathcal{D}}
\newcommand{\curve}{{\boldsymbol{c}}}
\newcommand{\ball}{{{\mathcal{B}}}}

\usepackage{enumitem}

\newcommand{\chain}{{\mathcal{C}}}
\newcommand{\ellSt}{\ell}
\DeclareMathOperator{\arsinh}{arsinh}
\DeclareMathOperator{\collar}{\mathscr{C}}

\title[low Steklov eigenvalues of finite volume hyperbolic surfaces]{Geometric bounds for low Steklov eigenvalues of finite volume hyperbolic surfaces}

\author[Hassannezhad]{Asma Hassannezhad}
\address{University of Bristol,
School of Mathematics,
Fry Building,
Woodland Road,
Bristol, 
BS8 1UG, U.K.}
\email{asma.hassannezhad@bristol.ac.uk}

\author[M\'etras]{Antoine M\'etras}
\address{Institut de Mathématiques, 
Université de Neuchâtel, 
Rue Emile-Argand 11,
2000 Neuchâtel, Suisse}
\email{antoine.metras@unine.ch}

\author[Perrin]{H\'el\`ene  Perrin}
\address{Institut de Mathématiques, 
Université de Neuchâtel, 
Rue Emile-Argand 11,
2000 Neuchâtel, Suisse}
\email{heleneperrin19@gmail.com}
\begin{document}

\maketitle

\begin{abstract}
We obtain geometric lower bounds for the low Steklov eigenvalues of finite-volume hyperbolic surfaces with geodesic boundary. The bounds we obtain depend on the length of a shortest multi-geodesic disconnecting the surfaces into connected components each containing a boundary component and the rate of dependency on it is sharp. Our result also identifies situations when the bound is independent of the length of this multi-geodesic. The bounds also hold when the Gaussian curvature is bounded between two negative constants and can be viewed as a counterpart of the well-known Schoen-Wolpert-Yau inequality for Laplace eigenvalues.  The proof is based on analysing the behaviour of the {corresponding Steklov} eigenfunction on an adapted version of thick-thin decomposition for hyperbolic surfaces with geodesic boundary. Our results extend and improve the previously known result in the compact case obtained by a different method.
 \end{abstract}

\section{Introduction}
Let $\Sigma$ be a connected finite volume hyperbolic surface with geodesic boundary, and let $b\ge1$ denote the number of boundary components. We consider 
the Dirichlet-to-Neumann map $\DtN$ 
\begin{eqnarray*}
    \DtN: C^\infty(\partial \Sigma) & \to & C^\infty(\partial \Sigma) \\
    f & \mapsto & \partial_\nu \tilde f,
\end{eqnarray*}
where $\tilde f$ is the harmonic extension of $f$ to $\Sigma$, and $\nu$ is the outward unit normal vector field along $\partial \Sigma$. By the standard spectral theory for self-adjoint operators, we have that its spectrum is discrete and each eigenvalue has finite multiplicity, see Section \ref{sec:per}.
    Let $0 = \sigma_0(\Sigma) < \sigma_1(\Sigma) \leq \cdots \leq \sigma_k(\Sigma) \leq \cdots \nearrow \infty$ be the sequence of its eigenvalue, also called the Steklov eigenvalues. The focus of this paper is on the study of geometric bounds for $\sigma_k(\Sigma)$ when $1 \leq k \leq b-1$ and $b > 1$.

Note that a lower bound for $\sigma_b(\Sigma)$ can be easily obtained by using the collar theorem and comparing $\sigma_b(\Sigma)$ with the $b$-th mixed Steklov-Neumann eigenvalue of a domain composed of a {union of disjoint half-collars about boundary geodesics}, and the lower bound will depend only on the length of the boundary (see e.g. \cite[Lemma 3]{Per22}). Hence, the study of the spectral gap and bounds on Steklov eigenvalues becomes an intriguing question when $k < b$. {We can refer to them as the \textit{low} Steklov eigenvalues.}

Lower bounds for the spectral gap of the Steklov problem on a compact Riemannian manifold with boundary have been studied by Jos\'e Escobar \cite{Esc97,Esc99}, and later by Pierre Jammes \cite{Jam15} where he obtained a Cheeger-type lower bound, see also \cite{HM}.
The question of obtaining more explicit geometric bounds for low Steklov eigenvalues has been recently studied in \cite{Per22,
Per23,HHH22,BBHM}.

 On a compact hyperbolic surface $\Sigma$,  H\'el\`ene Perrin \cite{Per22} obtained a geometric lower bound on a modified version of Cheeger-Jammes' constant in terms of the length of the shortest {multi-geodesic}
 separating $\Sigma$ into $k+1$ connected components, each of them containing at least one boundary component and showed that it is of great relevance in the estimate of low Steklov eigenvalues, {in particular, 
 she obtained lower and upper bounds for low eigenvalues of $\Sigma$ in terms of the length of {this multi-geodesic}.}
Let us state her result more precisely.  

For a given hyperbolic surface $\Sigma$, let $\chain_k$ be the set of multi-geodesics which consist of a union of disjoint simple closed geodesics, not intersecting $\partial \Sigma$, and dividing $\Sigma$ into $k+1$ connected components, each containing at least one connected component of $\partial \Sigma$. We define 
    \begin{align}\label{defn:ell}
        \ellSt_k := \inf_{\curve \in \chain_k} |\curve|,
    \end{align}
    where $|\curve|$ is the length of the {multi-geodesic} $\curve$. When $\chain_k=\emptyset$, we set $\ell_k=\infty$.
    
{For a compact hyperbolic surface $\Sigma$ of genus $g$  with $b$ geodesic boundary components of length $\leq2\arsinh(1)$, the
result in \cite{Per22} states that, assuming that $g\not = 0$ or $b> 3$,
there exists a constant $C_1$, depending only on $b$ and on $g$, and a universal constant $C_2$ such that for $1\leq k<\lceil \frac{b}{2}\rceil$ we have
\begin{equation}\label{thm:perrin}
C_1\ellSt_k^2\leq\sigma_k\leq C_2 \frac{\ellSt_k}{{\alpha}},
\end{equation}
{where $\alpha$ is the minimum length of geodesic boundary components.}
The inequality also holds for $\lceil \frac{b}{2}\rceil\leq k <b$, provided that $\chain_k\neq\emptyset$ {and $\ellSt_k$ is bounded above in terms of $g$ and $b$}.

{This result can be viewed as a counterpart of a result by Schoen, Wolpert and Yau \cite{SWY80} for  Laplace eigenvalues of a closed hyperbolic surface $\Sigma$.}
They showed that for $1 \leq k \leq 2g-3$, $\lambda_k$
is bounded above and below by positive constants (depending {only} on $g$ and $k$) times the length of the shortest multi-geodesic dividing $\Sigma$ into $k + 1$ connected components. There have been several studies on extending the Schoen-Wolpert-Yau inequality to noncompact surfaces and investigating the asymptotic behavior as the length of the multi-geodesic tends to zero in \cite{Dod,DR86,DRS87,burger88,burger90,GR19}.

In this article, we improve the power of $\ell_k$ in the lower bound of \eqref{thm:perrin} to achieve the optimal power as in the Schoen-Wolpert-Yau inequality.
Additionally, we generalise this lower bound by removing the upper bound on the maximum length of boundary components, {obtain a lower bound for all $k<b$}, and state the result in the context of noncompact finite volume hyperbolic surface.

\begin{theorem}\label{thm:main}
    Let $\Sigma$ be a finite volume hyperbolic surface with $b \ge 1$ geodesic boundary components. Let $\chi,g,p$ denote the Euler number of $\Sigma$, the genus and the number of cusps respectively, and let $\beta$ be the maximum length of the boundary components. We define
    \begin{align*}
        K := \begin{cases}
                b - 1 & \text{if $(g \geq 1 $ or $ p \geq 2)$ and $b \geq 1$}, \\
                b - 2 & \text{if $g = 0, p = 1$ and $b \geq 2$}, \\ 
                b - 3 & \text{if $g = 0, p = 0$ and $b \geq 3$}. 
        \end{cases}
    \end{align*}
    Then there exists a positive universal constant $C$ such that
          \begin{align*}
     \sigma_k(\Sigma)
 \ge \frac{C}{b|\chi|^3} \min\left\{\frac{1}{(1+\beta)^2e^\beta},\frac{\ellSt_k}{\beta}\right\},\qquad 0<k\le K,
    \end{align*}
    and
    \begin{align*}
        \sigma_{K+1}(\Sigma) \geq \frac{C}{b\chi^2(1+\beta)^2e^\beta}. 
    \end{align*}
    \end{theorem}
    As a consequence of a topological lemma (see Lemma \ref{lem:C_k}), we show that $\ell_k<\infty$ for every $1\le k \le K$. In particular, {since} there exists a surface for which  $\ell_k$  can be arbitrary small,  Theorem \ref{thm:main} shows that there is always a spectral gap between $\sigma_{K+1}$ and $\sigma_K$ when $\ell_K\to 0$.
    
         By combining {Theorem \ref{thm:main}} with the classical upper bound for Steklov eigenvalues of compact surfaces given in \cite{CGE,H11}, {which can be readily extended to the context of finite volume surfaces}, and with the upper bound in \cite{Per22} stated in (\ref{thm:perrin}) above {(see Remark \ref{rem:upperbound})}, we have that there exist positive  constants  $C_3=C_3(\chi,\beta)$, $C_4=C_4(\chi,\alpha)$, and $C_5=C_5(\chi,\beta)$ where $\alpha$ is the minimum length of the boundary components, such that
    \begin{align*}
       C_3 \min\{1,{\ellSt_k}\}\le  \sigma_k(\Sigma) \leq C_4 \min\{1,{\ellSt_k}\},\qquad 1\le k\le K 
    \end{align*}
    and
    \begin{align*}
       C_3\le \sigma_{K+1}(\Sigma) \leq C_5.
    \end{align*} 
   {We want to highlight here some special cases.}
   With the assumption that $\beta\leq 2\arsinh(1)$, by combining our result with Theorem 3 and Lemma 3 of \cite{Per22}, we have
    \begin{align}
    \frac{C_6}{b |\chi|^3}\min\{1,\frac{\ellSt_k}{\beta}\}\leq  \sigma_k(\Sigma) \leq C_7\min\{1, \frac{\ellSt_k}{\alpha}\},\qquad 1\le k\le K,
    \end{align}
    where $C_6$ and $C_7$ are universal constants.

  Assume that $\ell_k$ is bounded above in terms of a constant depending only on $\chi$. It is the case for example when 
  $k<\min \{\lceil\frac{b}{2}\rceil, K+1\}$ as shown in \cite{Per22}.
  {Then there exist positive constants $C_8(\chi, \beta)$ and $C_9(\beta)$ such that}
   \[C_8{\ellSt_k} \le \sigma_k\le C_9\frac{\ellSt_k}{
 \alpha
 },\]
{and $C_8$ and $C_9$ {can be} independent of $\beta$ when}
$\beta\le2\arsinh(1)$; this recovers an improved version of \eqref{thm:perrin} with an optimal  dependency on $\ell_k$.
Figure \ref{fig:large_ell_k} illustrates an example for which $\ell_k$ can be arbitrarily large for {$k=\lceil\frac{b}{2}\rceil$}. Hence, the above bound cannot hold for $k\ge\min \{\lceil\frac{b}{2}\rceil, K+1\}$ in general.

When $\ell_k$ tends to zero, {the combination of Theorem \ref{thm:main} and the upper bound in \cite{Per22} 
 as stated in \eqref{thm:perrin} (see Remark \ref{rem:upperbound}) 
}
implies that 
\begin{equation}\label{eq:lim}
  \frac{C}{b|\chi|^3\beta}\le\liminf_{\ell_k\to0}\frac{\sigma_k}{\ell_k}\le \limsup_{\ell_k\to0}\frac{\sigma_k}{\ell_k}\le \frac{C_2}{\alpha},\qquad 1\le k\le K,  
\end{equation}
where $C$ and $C_2$ are positive universal constants as mentioned above.
In particular, when $\beta=\alpha$ or $\alpha$ is a constant multiple of $\beta$, it shows the optimality of the power of $\beta$.
In general, when $\chi$  or $\beta/\alpha$ is large,  there will be a big gap between the upper and lower bound in inequality \eqref{eq:lim}. The study of the asymptotic of  $\frac{\sigma_k}{\ell_k}$ as $\ell_k\to0$ is an intriguing question. We refer to  \cite{Col85, burger88,GR19,Cha} for related studies for the Laplace eigenvalues.

It is also very interesting to investigate the optimality of the dependency of the lower bound on $\chi$. The power of $|\chi|$ in our lower bound for $\sigma_1$ obtained in Proposition~\ref{lem:ell_1} is $-2$. From \cite[Example 5.1]{BBHM}, we know that there exists a sequence of hyperbolic surfaces for which $\sigma_{b-1}$ decays at rate $\frac{1}{|\chi|}$.  It remains open whether the optimal power of $|\chi|$ is $-1$.  In the case of the Laplacian, Wu and Xue~\cite{YY} showed that the optimal power of $|\chi|$ in the lower bound for the first non-zero Laplace eigenvalue of a closed hyperbolic surface is $-2$.

The proof of Theorem \ref{thm:main} uses a different approach that the one used in \cite{Per22}. It is inspired by  Dodziuk-Randol's proof of  the Schoen-Wolpert-Yau inequality in \cite{DR86}, see also \cite{Dod,DRS87}. {Inspired by their approach,
we analyse the behaviour of Steklov eigenfunction on an adapted version of the thick-thin decomposition of a hyperbolic surface.  
A similar adaptation} is also used in \cite{BBHM} to obtain a geometric lower bound for the spectral gap in pinched negatively curved manifolds of dimension at least 3. However, the situation differs in dimension 2; unlike higher dimensions, the \textit{thick} part is not connected, presenting its own challenge.\\

Since the Steklov eigenvalues are invariant under any conformal change in the interior, Theorem \ref{thm:main} holds true for any Riemannian surface $\Sigma$ that is conformally equivalent to a hyperbolic surface with geodesic boundaries, with a conformal factor equal to 1 along the boundary.
Let $(\Sigma,h)$ be conformal to a hyperbolic surface {$(\Sigma,\bar h)$} with geodesic boundary and the conformal factor $f$, where $h=f\,\bar h$, satisfies  $A^{-1} \le f|_{\partial \Sigma}\le A$ for some constant $A\ge1$.  Then using the variational characterisation of the Steklov eigenvalues, we get
\[A^{{-1/2}}\, \sigma_k(\Sigma,\bar h)\le\sigma_k(\Sigma,h)\le A^{{1/2}}\sigma_k(\Sigma,\bar h).\]

{When $(\Sigma, h)$ is a compact Riemannian surface with Gaussian curvature in the interval $[-1, -\kappa]$ for some $\kappa > 0$, we show that it is conformal to a hyperbolic surface $(\Sigma, \bar{h})$ with geodesic boundary. Moreover, the conformal factor $f$ satisfies $1 \leq f \leq \kappa^{-1}$, ensuring that all results above remain valid, with ${\kappa}$ appearing as a multiplicative factor in the lower bound. For non-compact finite-volume surfaces with Gaussian curvature within the same range, similar bounds still hold. For further details, see Theorem \ref{rem:pinched curvature2}.\\
}

{In general}, we can ask  whether one can conformally deform a surface with boundary to obtain a hyperbolic surface with geodesic boundary while the conformal factor remains bounded. Uniformisation theorems for surfaces with boundary are studied in \cite{OPS88,Bre02a,Bre02b,Rup21}. {In particular, it is known that for a compact Riemannian surface  $(\Sigma,h)$ with boundary, 
 when} the integral of the geodesic curvature along $\partial \Sigma$ is non-negative, there exists a unique hyperbolic metric $\bar h=f h$ in the conformal class of $h$ such that the boundary of $(\Sigma,\bar h)$ is geodesic. The metric  $\bar h$ is called a \textit{uniform metric}.  However, the resulting surface may not be compact.

We can construct examples of a sequence of Riemannian surfaces $\Sigma_\epsilon$ with $\chi(\Sigma_\epsilon) < 0$, such that for any given $k \geq 1$, $\lim_{\epsilon \to 0} \sigma_k(\Sigma_\epsilon) = 0$. This sequence can be constructed by slightly modifying the example given in~\cite[Section 2.2]{GP10}, as illustrated in Figure \ref{fig:example}. This demonstrates that for $\epsilon$ small enough, $\Sigma_\epsilon$ cannot be conformally equivalent to a hyperbolic surface with a conformal factor equal to 1 on the boundary.
\begin{figure}[h]
    \centering
    \def\svgwidth{0.4\textwidth}
    %% Creator: Inkscape 1.3.2 (091e20ef0f, 2023-11-25, custom), www.inkscape.org
%% Accompanies image file 'small-sigma.pdf' (pdf, eps, ps)
\begingroup%
  \makeatletter%
  \providecommand\color[2][]{%
    \errmessage{(Inkscape) Color is used for the text in Inkscape, but the package 'color.sty' is not loaded}%
    \renewcommand\color[2][]{}%
  }%
  \providecommand\transparent[1]{%
    \errmessage{(Inkscape) Transparency is used (non-zero) for the text in Inkscape, but the package 'transparent.sty' is not loaded}%
    \renewcommand\transparent[1]{}%
  }%
  \providecommand\rotatebox[2]{#2}%
  \newcommand*\fsize{\dimexpr\f@size pt\relax}%
  \newcommand*\lineheight[1]{\fontsize{\fsize}{#1\fsize}\selectfont}%
  \ifx\svgwidth\undefined%
    \setlength{\unitlength}{245.2466349bp}%
    \ifx\svgscale\undefined%
      \relax%
    \else%
      \setlength{\unitlength}{\unitlength * \real{\svgscale}}%
    \fi%
  \else%
    \setlength{\unitlength}{\svgwidth}%
  \fi%
  \global\let\svgwidth\undefined%
  \global\let\svgscale\undefined%
  \makeatother%
  \begin{picture}(1,0.57378045)%
    \lineheight{1}%
    \setlength\tabcolsep{0pt}%
    \put(0,0){\includegraphics[width=\unitlength,page=1]{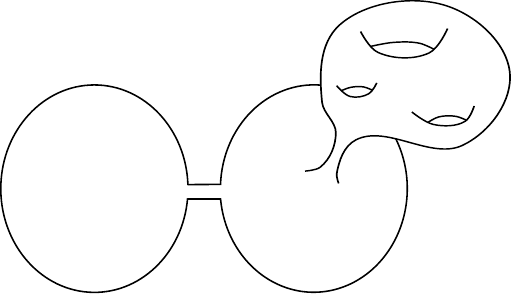}}%
    \put(0.27726191,0.18609022){\color[rgb]{0,0,0}\makebox(0,0)[lt]{\lineheight{1.25}\smash{\begin{tabular}[t]{l}$\epsilon^3$\end{tabular}}}}%
    \put(0,0){\includegraphics[width=\unitlength,page=2]{small-sigma.pdf}}%
    \put(0.37219247,0.1215267){\color[rgb]{0,0,0}\makebox(0,0)[lt]{\lineheight{1.25}\smash{\begin{tabular}[t]{l}$\epsilon$\end{tabular}}}}%
  \end{picture}%
\endgroup%

    \def\svgwidth{0.3\textwidth}
    \caption{$\Sigma_\epsilon$ is obtained as a union of two discs connected by a thin neck of length $\epsilon$ and width $\epsilon^3$, and removing a small disc around the centre of one of the discs, then performing a connected sum with a surface of genus at least $1$. }
    \label{fig:example}
\end{figure}
Moreover, by slightly modifying the example above, we can assume the geodesic curvature along $\partial \Sigma$ is non-negative. Hence, the conformal factor of the uniform metric along the boundary cannot remain uniformly bounded, and the answer to the question above is negative.\\

The paper is organised as follows. In Section 2, we cover some preliminaries, including the definition of an adapted version of the thick-thin decomposition for hyperbolic surfaces with geodesic boundary, and the result of Dodziuk and Randol \cite{DR86} on the behaviour of eigenfunctions in the thick and thin parts. Section 3 is devoted to the proof of the main result. We first present a topological lemma demonstrating when $\ell_k$ is achieved. Then we prove Theorem \ref{thm:main} for $k=1$ and show that the main theorem can be derived from this case.

\section*{Acknowledgement}
The authors would like to thank Bruno Colbois for useful discussions and interest in the project, {and are also grateful to the anonymous referees for their helpful comments.} A.\,H. and A.\,M. acknowledge support of EPSRC grant EP/T030577/1. A.\,M. also acknowledges support of the SNSF project ``Geometric Spectral Theory"  grant number 200021-19689.
\section{Preliminaries}\label{sec:per}
Throughout the paper, we assume that $\Sigma$ is a finite volume connected hyperbolic surface with nonempty geodesic boundary unless otherwise stated. We say that $\Sigma$ {is} of signature $(g,b;p)$ if it has genus $g$, $b$ geodesic boundary components, and $p$ cusps.

\subsection*{Steklov problem} When $\Sigma$ is compact, i.e. is of signature $(g,b;0)$, the Dirichlet-to-Neumann map $\DtN$ 
\begin{eqnarray*}
    \DtN: C^\infty(\partial \Sigma) & \to & C^\infty(\partial \Sigma) \\
    f & \mapsto & \partial_\nu \tilde f,
\end{eqnarray*}
where $\tilde f$ is the harmonic extension of $f$ to $\Sigma$, and $\nu$ is the outward unit normal vector field along $\partial \Sigma$, is a self-adjoint first-order elliptic pseudo-differential operator and its spectrum consists of a discrete sequence of non-negative real numbers with the only accumulation point at infinity, see e.g. \cite{LMP}. 

The Dirichlet-to-Neumann operator on  non-compact geometrically finite manifolds has been recently studied in \cite{Pol21}. However, in our setting, we explain that the discreteness of its spectrum is a consequence of classical theory.

Given a finite volume hyperbolic surface $\Sigma$, let $\double\Sigma$ denote the double of $\Sigma$ along its geodesic boundaries. It is a complete finite volume hyperbolic surface. 
We first briefly recall the spectral theory of the Laplace-Beltrami operator on a finite volume noncompact hyperbolic surface.  It is well-known that $\Delta$ is essentially self-adjoint and has a unique Friedrich extension. Its spectrum consists of a sequence of eigenvalues $0=\lambda_0<\lambda_1\le \lambda_2\le \cdots$ and continuous spectrum $[\frac{1}{4},\infty)$ (see e.g. \cite{Mul92}). There are finitely many eigenvalues in interval $[0,\frac{1}{4})$.  The multiplicity of $0$ is 1 and corresponds to the constant functions.

Let us consider the Dirichlet Laplacian on $\Sigma$. Then we immediately get that the bottom of the Dirichlet spectrum 
\[\lambda_0^D(\Sigma)=\inf_{0\neq f\in H^1_0(\Sigma)}\frac{\int_M|\nabla f|^2}{\int_Mf^2}\]
is strictly positive because $\lambda_0^D(\Sigma)\ge\min\{1/4,\lambda_1(\double\Sigma)\}$. It implies that for any $f\in C^\infty(\partial \Sigma)$, there is a unique square integrable harmonic extension $\tilde f$ to $\Sigma$. 
Hence, the Dirichlet-to-Neumann map $\DtN$ is well defined.  It is symmetric and positive.  We consider its Friedrich extension, also denoted by $\DtN$. 
The map $\DtN$ is a first-order elliptic operator and its spectrum consists of a discrete sequence of non-negative real numbers with the only accumulation point at infinity. The discreteness of the spectrum follows from the compactness of the trace operator $T:H^1(\Sigma)\to L^2(\partial \Sigma)$.
 The eigenvalues of the Dirichlet-to-Neumann map are the same as  the eigenvalues of the {Steklov problem}:
\[\begin{cases}
    \Delta u = 0 & \text{in $\Sigma$},\\
    \partial_\nu u = \sigma u &\text{on $\del \Sigma$},
\end{cases}\]
 We enumerate them in increasing order counting their multiplicities:
\begin{equation*}0 = \sigma_0 < \sigma_1 \le \sigma_2 \le \cdots \nearrow\infty.\end{equation*}
We have the following variational characterisation of the Steklov eigenvalues. 
\begin{equation}\label{eq:var}
    \sigma_k=\inf_{V_{k+1}}\sup_{{0\neq} f\in V_{k+1}}\frac{\int_\Sigma |\nabla f|^2}{\int_{\partial\Sigma} f^2},\end{equation}
where $V_{k+1}$ is a $(k+1)$-dimensional subspace of $H^1(\Sigma)$.

In the subsequent sections, we also consider the mixed Steklov-Dirichlet and mixed Steklov-Neumann eigenvalue problems, where either the Dirichlet or Neumann boundary condition is assumed on a portion of the boundary.

{The variational characterization of the Steklov-Neumann and Steklov-Dirichlet eigenvalues is similar to that of the Steklov eigenvalues. The only difference is that the integration in the denominator of {the Rayleigh quotient in} \eqref{eq:var} is restricted to the Steklov part of the boundary. For the Steklov-Dirichlet problem, we should also restrict the functional space to those functions that vanish on the part of the boundary with the Dirichlet condition.  }

\subsection*{Thick-thin Decomposition.}
We define the thick-thin decomposition of $\Sigma$ as follows. Let $\double\Sigma$ be the double of $\Sigma$ along its totally geodesic boundary.
{We define $(\double\Sigma)_{\thin}$ to be the subset of $\double\Sigma$ consisting of}
\begin{itemize}
    \item[(i)] 
the union of collars $\collar(\gamma)$ for all simple closed geodesic of length $\le 2\arsinh(1)$:
\[\collar(\gamma)=\left\{p\in\double\Sigma\,|\, \dist(p,\gamma)\le \tube(\gamma)\right\},\quad \text{where}\quad \tube(\gamma)=\arsinh\left(\frac{1}{\sinh(|\gamma|/2)}\right)\]
 and $|\gamma|$ denotes the length of $\gamma$. The collar $\collar(\gamma)$ is isometric to the warped product $[-\tube(\gamma),\tube(\gamma)]\times_{\cosh}\mathbb{S}^1_{R}$, where $2\pi R=|\gamma|$. {Recall that the warped product $[-\tube(\gamma),\tube(\gamma)]\times_{\cosh}\mathbb{S}^1_{R}$ is the Riemannian surface $[-\tube(\gamma),\tube(\gamma)]\times\mathbb{S}^1_{R}$ with Riemannian metric $\dd t^2+\cosh^2( t) g_{\mathbb{S}^1_{R}}$ where $\dd t^2$ and $g_{\mathbb{S}^1_{R}}$ denote the canonical metrics of $[-\tube(\gamma),\tube(\gamma)]$ and $\mathbb{S}^1_{R}$}; 
\item[(ii)] a finite collection of cusps $\cusp$ isometric to {the warped product} $(-\infty, \log 2)\times_f\mathbb{S}^1$, where $f(t)=e^{t}$.
\end{itemize} 

According to the Collar Theorem \cite[Theorem 4.4.6]{Bus92}, the collars and cusps are mutually disjoint and the injectivity radius of any point in the complement of $(\double\Sigma)_{\thin}$ is strictly bigger than $\arsinh(1)$. 
{By defining 
\[(\double\Sigma)_\thick:=\{p\in \double\Sigma\,|\, \inj_p(\double\Sigma)>\arsinh(1)\},\]
{we have a covering of $\double \Sigma$ by $(\double\Sigma)_{\thick}$ and $(\double\Sigma)_{\thin}$} which we call the \textit{thick-thin decomposition}.} 
Note that $ (\double\Sigma)_{\thick}$ is always nonempty, and {if $(\double\Sigma)_{\thin}\neq\emptyset$}, the intersection of the thick and thin parts is nonempty. {Indeed, for any point on $p\in \partial\collar(\gamma)\subset (\double\Sigma)_{\thin}$, by \cite[Theorem 4.1.6]{Bus92}), we have $\inj_p(\double\Sigma)>\arsinh(1)$. }

{We can view $\Sigma$ as a subdomain of $\double\Sigma$ and define its thick-thin decomposition by considering the intersection of $\Sigma$ with the thick and thin parts of $\double \Sigma$. {However, we {shall need to} consider an alternative definition described below.}}

{We observe that} the intersection between $\Sigma$ and a collar $\collar(\gamma){\subset(\double\Sigma)_{\thin}}$ 
is the collar itself if $\gamma\subset \Sigma$; it is a half-collar if $\gamma$ is one of the boundary components. But if there is at least one geodesic boundary of $\Sigma$ of length $>2\arsinh(1)$, we may have that $\gamma\cap \Sigma$ is a geodesic arc with endpoints on one or two boundary components (as it happens in Figure \ref{fig:decomposition} where a geodesic of the decomposition intersect the geodesic boundary).
For technical reasons, we want to avoid this situation. Hence, {to define the thick-thin decomposition for $\Sigma$, we first modify} the definition of the thick-thin decomposition of {$\double\Sigma$}.

Let $\{B_1,\cdots, B_b\}$ be the boundary components of $\partial \Sigma$ and $\beta=|B_{\max}|=\max_i|B_i|$. We take 
\begin{equation}\label{epsilon}
\eps=\eps(\beta)=\min\left\{{\arsinh(1)}, \tube(B_{\max})\right\},\end{equation} and define {the $\eps$-\textit{thick-thin decomposition}} of $\double\Sigma$ as follows.
$$(\double\Sigma)_{\thick}^\eps=\{p\in \double \Sigma\,|\, \inj_p(\double\Sigma)>{\eps}\},\qquad (\double\Sigma)_{\thin}^\eps=\bigcup_{|\gamma|\le 2\eps}\collar(\gamma)\bigcup \left(\bigcup_j \cusp_j\right).$$
The union of the $\eps$-thick and $\eps$-thin parts cover the whole $\double\Sigma$ because if $p\in \double\Sigma \setminus \double\Sigma_{\thin}^\eps$, either $p\in\collar(\gamma)$ for a $\gamma$ with $2\arsinh(1)\geq|\gamma|\geq2\eps$ which implies $\inj_p(\double\Sigma)\geq\eps$, or $p \in \double\Sigma_{\thick}$ and $\inj_p(\double\Sigma)>\arsinh(1)\geq \eps$. 
{We now define the $\eps$-thick-thin decomposition of $\Sigma$ as follows. }
{$$ \Sigma_{\thick}^\eps:=(\Sigma\cap(\double\Sigma)_{\thick}^\eps)\setminus (\cup_j\collar^+_j)^{\circ},\qquad \Sigma_{\thin}^\eps:=((\double\Sigma)_{\thin}^\eps\cap \Sigma)\cup (\cup_j\collar^+_j)),$$
where $\collar_j^+$ denote the half-collar around the geodesic boundary $B_j$ and $(\collar_j^+)^{\circ}$ its interior. We note that for any point $p \in\Sigma_{\thick}^\eps$, $\inj_p(\Sigma)\geq\eps$ because $\inj_p(\double \Sigma)>\eps$ and $\dist(p,\partial \Sigma)\geq\eps$. We also note that $\Sigma_{\thin}^\eps$ is a disjoint union of collars, half-collars, and cusps.
}

\begin{figure}[h]
    \centering
    \def\svgwidth{0.6\textwidth}
    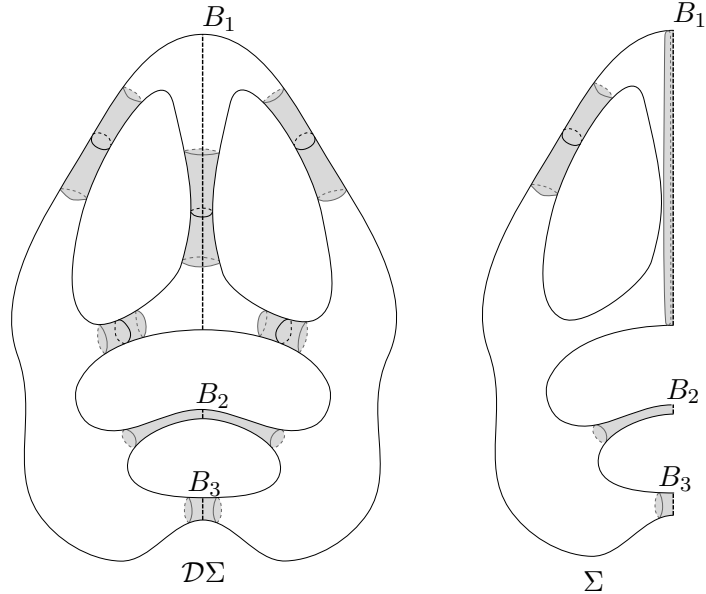
    \caption{On the left, the grey parts show  $(\double\Sigma)_{\thin}$ where $\Sigma$ is a hyperbolic surface with 3 boundary components $B_1, B_2$ and $B_3$, and on the right, the thin part $\Sigma_{\thin}^{\eps}$ of $\Sigma$. Note that since $\eps \leq \arsinh(1)$, some of the original thin tubes are no longer in the thin part. Furthermore, by definition, the half-collar of each boundary component is part of $\Sigma_{\thin}^{\eps}$.} 
    \label{fig:decomposition}
\end{figure}

We shall see that a key ingredient of the proof is the behaviour of the Steklov eigenfunctions on {the half-collars,}  the thick part and the thin collars, while the presence of the cusps will not be important. We list two key lemmas due to Dodziuk and Randol \cite{DR86,Dod} which provide estimates on the Dirichlet energy of the Steklov eigenfunction on the thick part and on thin collars. 
\begin{lemma}[\cite{DR86}] \label{lem:tube-energy}
    Let $\collar(\gamma)\subset (\double \Sigma)_{\thin}$ and $\Gamma_1$ and $\Gamma_2$ be the two boundary components of $\collar(\gamma)$.  Let $f$ be a differentiable function on $\collar(\gamma)$ and assume that there exists a positive constant $c>0$ such that 
    \begin{align*}
        \min_{x \in \Gamma_1} |f(x) - f(x^*)| \geq c, 
    \end{align*}
    where $x^* \in \Gamma_2$ is the  reflection of $x \in\Gamma_1$ with respect to $\gamma$.
    Then
    \begin{align*}
        \int_{\collar(\gamma)}|\nabla f|^2  \geq \frac{c^2 |\gamma|}{4}.
    \end{align*}
\end{lemma}
The second key lemma shows that the oscillation of a Steklov eigenfunction on the thick part is bounded above by the corresponding Steklov eigenvalue.  
\begin{lemma}\label{lem:sobinq}
If $\varphi$ is a $\sigma$-Steklov eigenfunction with $\|\varphi\|_{L^2(\partial \Sigma)}=1$, then
for any $x,y$ that belongs to a single connected component $\Sigma_0$ of ${\Sigma^{\eps}_{\thick}}$,  we have 
\begin{equation}\label{cor:DR}
    |\varphi(x)-\varphi(y)|\le c(\beta) \sqrt{\sigma |\Sigma_0|},
\end{equation}
where $c(\beta)=c\,\eps^{-1}$ for some positive universal constant $c$. Note that $c(\beta)$ depends on $\beta=|B_{\max}|$ only when $\beta\geq 2\arsinh(1)$, {otherwise it is independent of $\beta$.}
\end{lemma}
\begin{proof} {This result is a consequence of the fact that} 
there exists a positive universal constant $c_1$ such that for any harmonic function $\varphi$  on $\Sigma$ 
 and for any ball $\ball$ {centered at a point $x\in\Sigma^{\eps}_{\thick}$ of radius $0<r\le\eps$}
 we have

\begin{equation}\label{lem:DR}
\|\nabla\varphi\|_{\infty,\ball/2}\le  c_1r^{-1}\left(\int_\ball|\nabla \varphi|^2\dd A\right)^{\frac{1}{2}},
\end{equation}
where  $\ball/2$ is the ball concentric with $\ball$ and of half the radius of $\ball$.
We refer to \cite{DR86} and \cite[page 32]{Dod} for the proof of inequality \eqref{lem:DR}. See also \cite[Section 4]{BBHM}. The proof of its consequence, {inequality} \eqref{cor:DR}, can be also found in \cite{DR86,BBHM} but for the convenience of the reader, we add the details of the proof here.

Let $\{\ball_j\}_{j=1}^N$, $N\in \mathbb{N}$ be a chain of overlapping balls centered at $x_j\in \Sigma_0 \subset  \Sigma^\eps_{\thick}$, and of radius $r_j=\eps/2$, connecting $x$ and $y$ such that $\{\ball_j/2\}_{j=1}^N$ are mutually disjoint.  Then $N$ can be bounded above by $c_2\eps^{-2}|\Sigma_0|$. We can also assume that each ball $2\ball_j$  intersects at most $c_3$ balls, where {$c_2$ and $c_3$ are}  universal constants. Then
 \begin{eqnarray*}
   \sum_j  \|\nabla \varphi\|_{\infty, \ball_j }&\le & c_1\eps^{-1}\sum_{j=1}^N\left(\int_{2\ball_j}|\nabla\varphi|^2\right)^{1/2} \\
   &\le& c_1\eps^{-1} \sqrt{N}\left(\sum_{j=1}^N\int_{2\ball_j}|\nabla\varphi|^2\right)^{1/2}\\
   &\le& c_4\eps^{-2}  \sqrt{|\Sigma_0|}\left(\int_{\Sigma}|\nabla\varphi|^2\right)^{1/2}\\
   &=&c_4\eps^{-2}  \sqrt{|\Sigma_0|\sigma},
  \end{eqnarray*}
  where $c_4$ is a universal constant depending on $c_1,c_2$, and $c_3$.\\
Now, let  $\curve:[0,1]\to \Sigma_0$ be a piece-wise geodesic  curve connecting $x$ and $y$. We choose the partition $0=t_0<t_1< \cdots<t_m=1$ such that $\curve|_{[t_j,t_{j+1}]}$ is a geodesic and $\curve([t_j,t_{j+1}])\subset \ball_{j}$. Hence, the length of $\curve|_{[t_j,t_{j+1}]}$ is bounded above by $2\eps$. We conclude that
   \begin{eqnarray}\label{oscilation}
      \nonumber |\varphi(x)-\varphi(y)|&=&\left|\int_0^1\frac{d}{dt}\varphi\circ \curve(t)dt\right|\\
       \nonumber&\le& \sum_{j}\|\nabla \varphi\|_{\infty, B_j }\int_{t_j}^{t_{j+1}}|\curve'(t)|dt\\
      \nonumber &\le&2\eps\sum_{j}\|\nabla\varphi\|_{\infty, B_{j} }\\
      &\le& c \eps^{-1}\sqrt{|\Sigma_0|}\sigma,
   \end{eqnarray}
   where $c$ is a universal constant.
   \end{proof}

\section{Proof of the main result} 
Let us recall the definition of  $\ellSt_k := \inf_{\curve \in \chain_k} |\curve|,$
    where $\chain_k$ denotes the set of multi-geodesics formed by disjoint simple closed geodesics, dividing $\Sigma$ into $k+1$ connected components, each containing at least one part of $\partial \Sigma$. When $\chain_k=\emptyset$, we set $\ell_k=\infty$.
    The following lemma shows when   $\ellSt_k$ will be achieved.\\

\begin{lemma} \label{lem:C_k}
    Let $\Sigma$ be a hyperbolic surface of signature (g,b;p). Then
    \begin{itemize}
    \item[(i)] $\chain_b=\emptyset$ for any signature. In particular, $\chain_1=\emptyset$ when $b=1$.
    \item[(ii)] $\chain_1 = \emptyset$  when $(g,b;p)=(0,3;0)$ or $(0,2;1)$.
        \item[(iii)]
    {$\chain_{b-1}\neq \emptyset$ when $(g \geq 1\text{~or~} p\ge 2)$ and $b \geq 2$.}

        \item[(iv)] $\chain_{b-2} \neq\emptyset$ and $\chain_{b-1} = \emptyset$ for any surface with signature $(0,b;1)$, $b \geq 3$.
        \item[(v)] $\chain_{b - 3}\neq\emptyset$  and $\chain_{b-2} = \emptyset$ for any surface with signature $(0,b;0)$, $b \geq 4$.
    \end{itemize}
\end{lemma}

\begin{proof} 
\begin{itemize}
    \item[$(i)$] It is clear that $\chain_b = \emptyset$ as that would require finding a decomposition of $\Sigma$ into $b+1$ components each containing part of the $b$ components of the boundary. 
    \item[$(ii)$] In this case, the surface is either a pair of pants or a surface with two boundary components and one cusp. In either case, one can see that $\chain_1 = \emptyset$ as there are no geodesic loops (non-homotopic to the boundary components). 
    \item[$(iii)$] To show that $\chain_{b-1}$ is non-empty, we proceed as shown in Figure \ref{fig:cut-genus}. Note that the curves are not in the same free homotopy class and can be viewed as geodesics. Indeed, if two nontrivial closed curves are disjoint, then the closed geodesics in their respective free homotopy class either coincide as point sets or remain disjoint (see e.g. \cite[Chapter 1]{Bus92}).  

\begin{figure}[h]
    \centering
    \def\svgwidth{0.3\textwidth}
    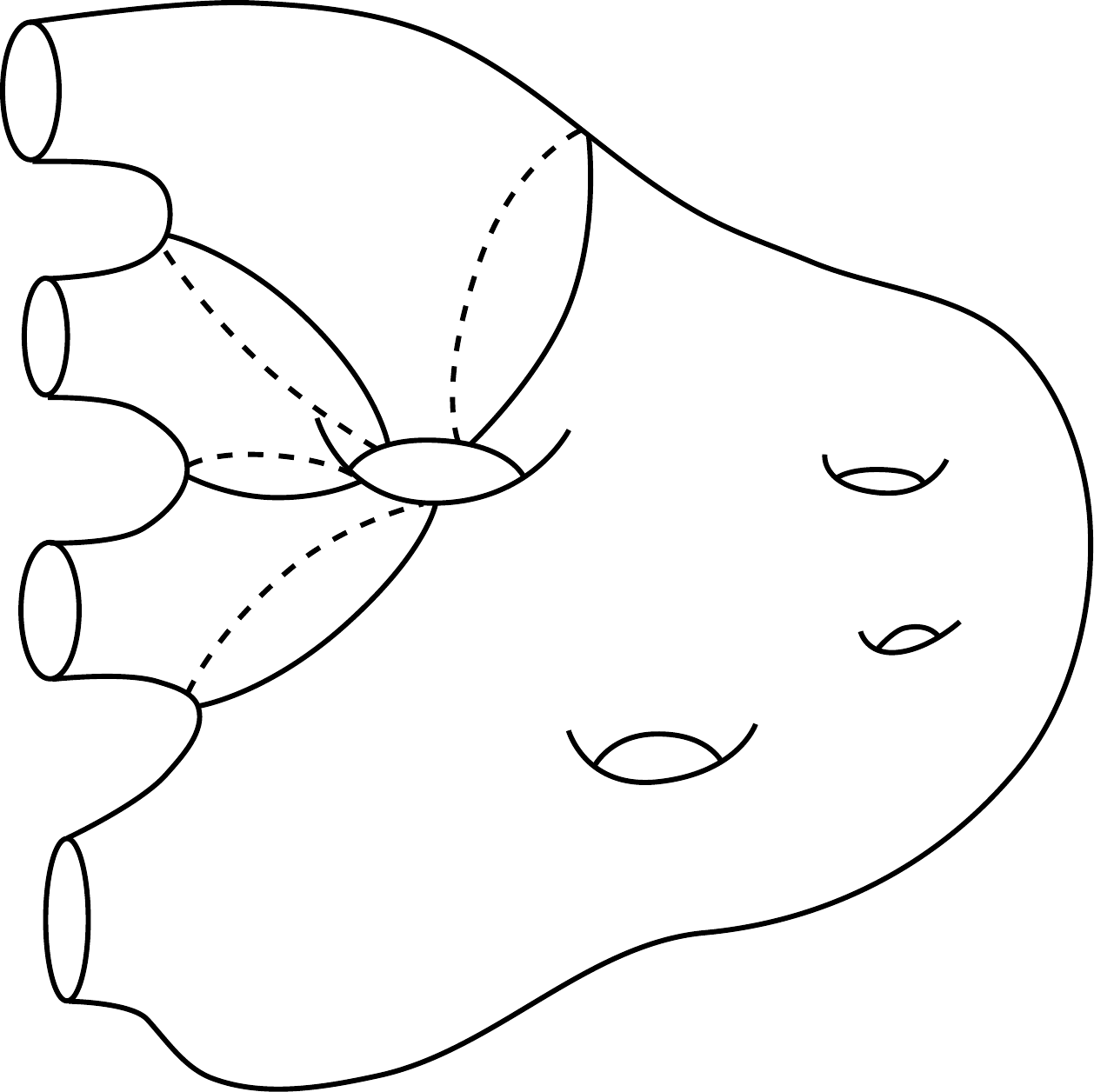
    \def\svgwidth{0.3\textwidth}
    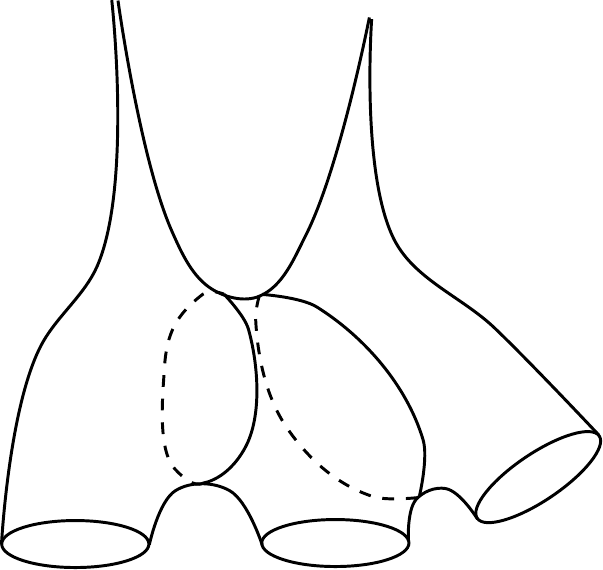
    \caption{Examples of decomposition when the genus is non-zero or there are at least two cusps.}
    \label{fig:cut-genus}
\end{figure}

\item[$(iv)$ \& $(v)$] The cases $(iv)$ and $(v)$ are similar in their proof, we first prove case $(v)$. Let $B_1$ and $B_2$ be two boundary components  and $\double_{B_1,B_2}\Sigma$ be a surface obtained by the doubling of $\Sigma$ along these two boundary components (as illustrated in Figure \ref{fig:cut-nogenus}). Then $\double_{B_1,B_2}\Sigma$ is a surface of signature $(1,2(b-2);0)$. Applying the proof of part $(iii)$, we can obtain a multi-geodesic $\chain$ decomposing $\double_{B_1, B_2}\Sigma$ into $2b - 4$ components and such that $B_1$ and $B_2$ are part of $\chain$. This chain when restricted to $\Sigma$ decomposes $\Sigma$ into $b - 2$ components, each having part of the boundary $\del \Sigma$. Hence $\chain|_\Sigma \in \chain_{b-3} \neq \emptyset$. A similar reasoning shows that $\chain_{b-2} = \emptyset$. For the case $(iv)$, the proof is similar but we double along a single boundary component to obtain a surface of signature $(0, 2(b-1); 2)$. 
\end{itemize}
\end{proof}

\begin{figure}[h]
    \centering
    \def\svgwidth{0.6\textwidth}
    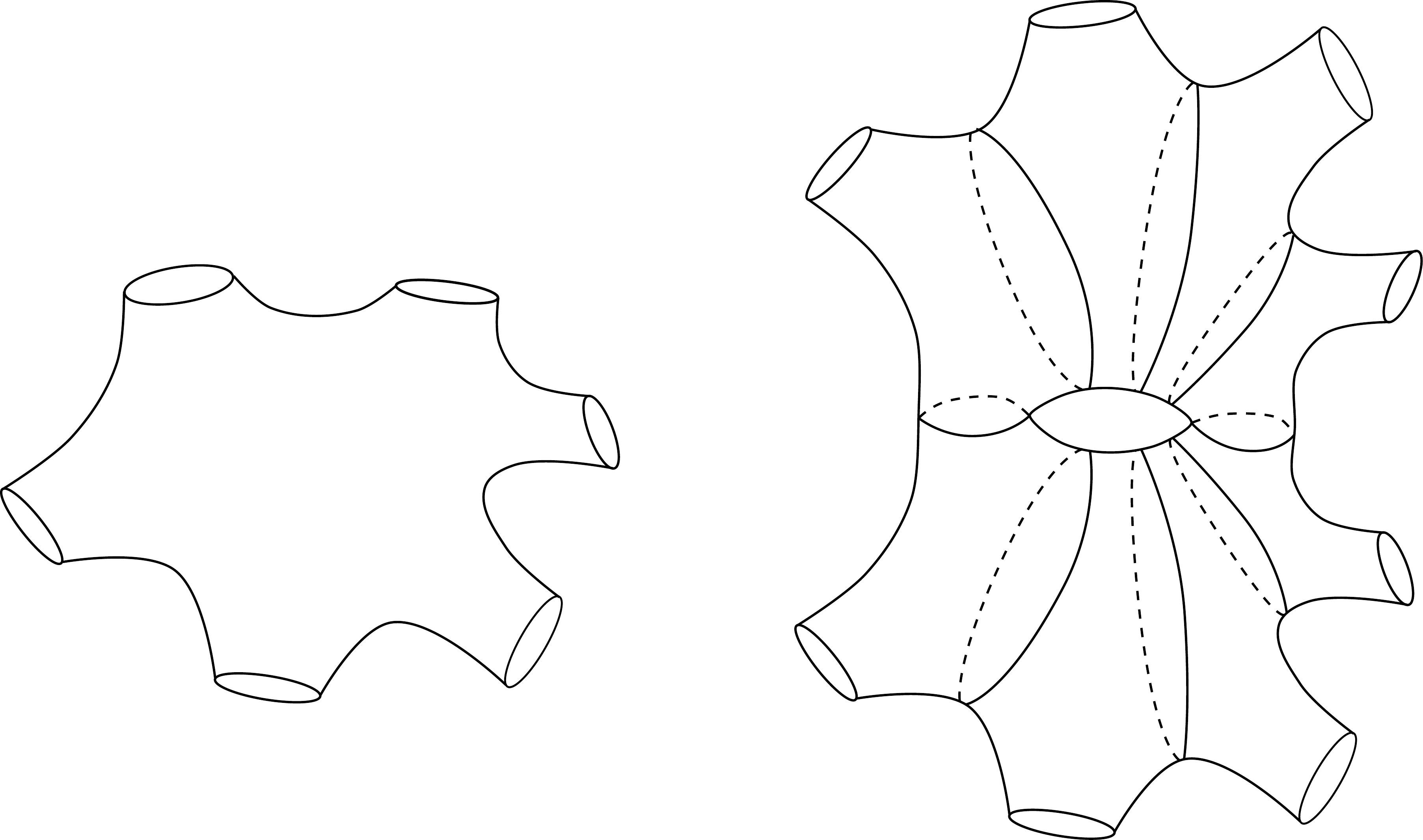
    \caption{A surface with 6 boundary components and genus 0 and how to obtain 4 disjoint components each containing part of the boundary by considering its double with respect to two of the boundary components.}
    \label{fig:cut-nogenus}
\end{figure}

\begin{remark}\label{rem:ellk bound}
When $\beta < 2 \arsinh(1)$ and $k < \min \left\{\lceil \frac{b}{2} \rceil, K+1\right\}$, $\ell_k$ cannot be arbitrarily large. Indeed, as shown in \cite{Per22}, it follows from Bers' theorem that in this case, $\ell_k$ is bounded in terms of an explicit constant depending only on the genus, number of cusps and number of boundary components. Note that even though the result given in \cite{Per22} is for compact surfaces, one can easily extend it to allow for cusps by using an appropriate generalisation of Bers' theorem (see e.g. \cite{Bus92} or \cite[Theorem 6.10 and its proof]{BPS}).
On the other hand, for $k \geq \min\{\lceil \frac{b}{2} \rceil, K+1\}$, one can construct surfaces making $\ell_k$ arbitrarily large. For example,  Figure \ref{fig:large_ell_k} illustrates that one can make $\ell_3$ arbitrarily large while keeping the length of the boundary components constant.
This behaviour contrasts with the one observed for the equivalent $\ell_k$ used in the Laplacian eigenvalue problem. In this case, it is always bounded from above by some constant depending on the genus and number of cusps, a consequence of the bound on lengths of pants decomposition \cite{Bus92, BPS}.
\end{remark}

\begin{figure}[h]
    \centering
    \def\svgwidth{0.4\textwidth}
    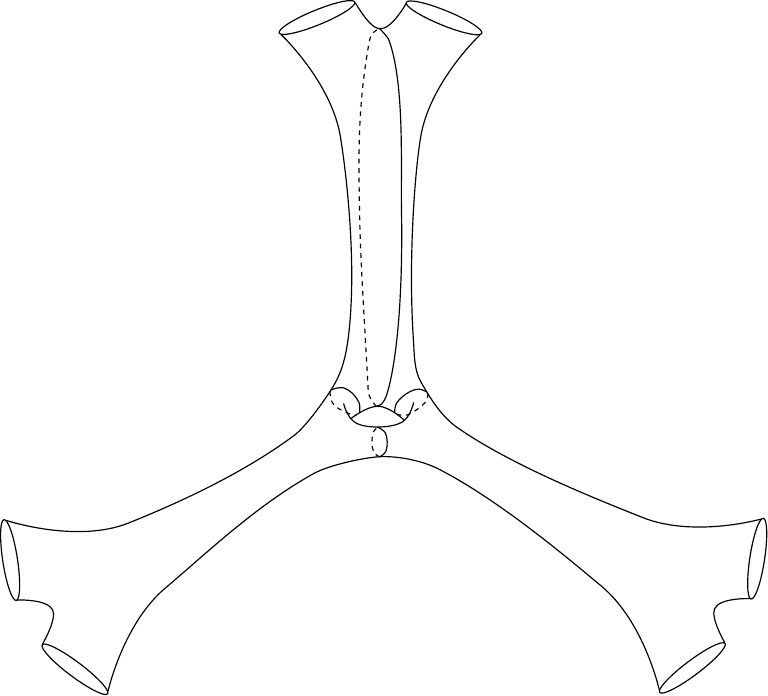
    \caption{Example of surface {with $6$ boundary components and} a large $\ell_3$.}
    \label{fig:large_ell_k}
\end{figure}

\begin{proposition} \label{lem:ell_1}
    Let $\Sigma$ be a hyperbolic surface of signature $(g,b;p)$. Let $\beta$ denote the maximum length of the boundary components and $\chi$ the Euler number of $\Sigma$. Then there exists a universal positive constant $c$  such that 
   \[\sigma_1(\Sigma) \geq \frac{c}{b\chi^2} \min\left\{\frac{1}{(1+\beta)^2e^\beta},\frac{\ellSt_1}{\beta}\right\}.\]
\end{proposition}

\noindent{Let us first give an outline of the proof.\\ }

\noindent{\textit{Sketch of the proof.} We consider the behaviour of a normalised $\sigma_1$-eigenfunction $f$ on $\eps$-thick and thin parts.  We first analyse the behaviour of $f$ on the $\epsilon$-thin part adjacent to the boundary where it has the largest $L^2$-norm.
If $f$ is `almost' $L^2$-orthogonal to 1 along that boundary component, then we can modify it and use it as a test function to compare $\sigma_1$ and the first non-zero Steklov-Neumann eigenvalue on that half-collar, obtaining a lower bound for $\sigma_1$ depending only on the length of that boundary component.
Otherwise, since $f$ is $L^2$-orthogonal to 1 on $\partial \Sigma$, it should have a `large' variation somewhere. Either this variation occurs on some $\epsilon$-thin parts adjacent to the boundary, leading to a lower bound for $\sigma_1(\Sigma)$ independent of $\ell_1$, or it occurs away from the boundary. In the latter case, Lemma \ref{lem:sobinq} tells us that the variation on the $\epsilon$-thick part is controlled in terms of $\sigma_1$. Thus the `large' variation must happen in the $\eps$-thin parts in the interior, composed of collars around short geodesics. Lemma \ref{lem:tube-energy} then relates the Dirichlet energy and the length of these short geodesics which ultimately gives the link between $\sigma_1(\Sigma)$ and $\ell_1$. We now proceed with the details of the proof.}
\begin{proof}[Proof of Proposition \ref{lem:ell_1}]
Let $\collar^+_1,\ldots,\collar^+_b$ denote the half-collars around $B_1,\ldots,B_b$, and $\tube_1,\dots,\tube_b$ be their corresponding width ($\tube_i=\tube(B_i))$. For each $\collar^+_j$ we consider the Fermi coordinates $(t,s)$ based on $B_j$, where $0\le t\le \tube_j$ and $0\le s\le |B_j|$. The Riemannian metric on $\collar_j^+$ in these coordinates is given by $\dd t^2+\cosh^2(t)\dd s^2$.

Let $f$ be an eigenfunction associated to the first non-zero Steklov eigenvalue $\sigma_1(\Sigma)$ such that $\int_{\partial \Sigma}f^2=1$. { We denote by $\bar f_j$ the average value of $f$ over $B_j$ and by $\bar f_j(\tube_j)$ the average of $f(\tube_j,\cdot)$ over $B_j$: 
\[
\bar f_j=\frac{1}{|B_j|}\int_{B_j}f(0,s)\dd s,\quad \bar f_j(\tube_j)=\frac{1}{|B_j|}\int_{B_j}f(\tube_j,s)\dd s.
\]
Since $\int_{\partial \Sigma}f^2=1$, there exists a boundary component $B_I$ such that $\int_{B_I}f^2 \geq 1/b$. 
Without loss of generality, we assume $\bar f_I\geq 0$. Note that $\bar f_I\le \frac{1}{\sqrt{|B_I|}}\|f\|_{L^2(B_I)}$. We now prove the result considering two separate cases: when $0\le \bar f_I<\frac{1}{{\sqrt{2b|B_I|}}}$, and when $\bar f_I \geq\frac{1}{\sqrt{2b|B_I| }}$.} Note that if $b=1$, since $\bar f_I=0$ the first case is the only one one needs to consider.\\

\noindent{\textbf{Case 1. $0\le \bar f_I<\frac{1}{{\sqrt{2b|B_I|}}}$.}
We define $ \tilde f:=f-\bar f_I$.
We note that $\int_{B_I}\tilde f =0$, and
\[
\int_{B_I}\tilde f ^2= \int_{B_I}(f-\bar f_I)^2=\int_{B_I}f^2-|B_I|\bar f_I^2\geq \frac{1}{b}-|B_I|\bar f_I^2\ge \frac{1}{2b}.
\]
Hence, we get
\begin{equation*}
\sigma_1(\Sigma)=\int_{\Sigma}|\nabla f|^2\geq \int_{\collar^+_I}|\nabla f|^2=\int_{\collar^+_I}|\nabla \tilde f|^2\geq \frac{1}{2b}\frac{\int_{\collar^+_I}|\nabla \tilde f|^2}{\int_{B_I}\tilde f^2}\geq \frac{1}{2b} \sigma_1^N(\collar^+_I),
\end{equation*}
where $\sigma_1^N(\collar^+_I)$ is the first non-zero mixed Steklov-Neumann eigenvalue on $\collar^+_I$ with Steklov condition on $B_I$ and Neumann condition on the other boundary component of $\collar^+_I$. The last inequality follows from the variational characterization of $\sigma_1^N(\collar^+_I)$. Steklov eigenvalues of warped products in dimension 2 such as the half-collar $\collar^+_I$ are explicitly computable. To obtain the value of $\sigma_1^N(\collar^+_I)$, we observe that the metric $\dd t^2+\cosh^2(t)\dd s^2$ is conformal to $\frac{1}{\cosh^2(t)}\dd t^2+\dd s^2$, that is $\collar^+_I$ is conformal to the right cylinder $[0, \arctan(\sinh(\tube))]\times \mathbb{S}^1_{R}$ whose mixed Steklov-Neumann eigenvalues are obtained using the method of separation of variables. Since conformal surfaces with conformal factor equal to one along the boundary are Steklov isospectral, we have }

\[
\sigma_1^N(\collar^+_I)=\frac{2\pi}{|B_I|}\tanh\left (\frac{2\pi}{|B_I|}\arctan(\frac{1}{\sinh(\frac{|B_I|}{2})})\right ).
\] Therefore,
\[
\sigma_1(\Sigma)\geq \frac{1}{2b}\sigma_1^N(\collar^+_I)\ge \frac{c_0}{b|B_I|}\min\left\{1,\frac{1}{|B_I|e^{|B_I|}}\right\}\ge\frac{c_0}{b\beta{(1+\beta)} e^{\beta}},
\]
for some positive universal constant $c_0$.

{\noindent\textbf{Case 2. $\bar f_I \geq\frac{1}{\sqrt{2b|B_I| }}$.}  
{We first show that $\sigma_1$ is bounded below by the absolute value of the difference between $\bar f_I$ and $\bar f_I(\tube_I)$.}
We have
\begin{equation}\label{eq2:l2normfi}
\begin{split}
\sigma_1(\Sigma)&=\int_{\Sigma}|\nabla f|^2\geq\int_{\collar^+_I}|\nabla f|^2 {=\int_{B_I}\int_0^{\tube_I}\left(|\partial_t f|^2+\cosh(t)^{-2}|\partial_sf|^2\right)\cosh t \dd t\dd s}\\
&\ge\int_{B_I}\int_0^{\tube_I} (\partial_t f)^2(t,s)\cosh(t) \dd t {\dd s}\\
&\geq \frac{1}{\int_0^{\tube_I} \frac{1}{\cosh(t)} \dd t}\int_{B_I} \left( \int_0^{\tube_I} \partial_tf \dd t \right)^2 \\
&\geq \frac{2}{\pi} \int_{B_I}(f(\tube_I,s)-f(0,s))^2 \dd{s}\\
&= \frac{2}{\pi}\| f_I(\tube_I,\cdot)-f(0,\cdot)\|_{L^2(B_I)}^2,
\end{split}
\end{equation}
{where the inequality between the second and the third lines is obtained by using Cauchy-Schwarz inequality.}
On the other hand,  
 we have
\begin{equation}
\begin{split}\label{eq2:bar fi}
{|}\bar f_I-\bar f_I(\tube_I){|}&\le \frac{1}{|B_I|}\|f_j(\tube_j,\cdot)-f(0,\cdot)\|_{L^1(B_I)}\\
   &\le \frac{1}{\sqrt{|B_I|}}\|f_I(\tube_I,\cdot)-f(0,\cdot)\|_{L^2(B_I)}.
\end{split}
\end{equation}
 Combining \eqref{eq2:l2normfi} and \eqref{eq2:bar fi}, we get
\begin{equation}\label{eq2:fi-fiwi}
    \bar f_I-\bar f_I(\tube_I)\le  |\bar f_I-\bar f_I(\tube_I)|\le \sqrt \frac{\pi}{{2|B_I|}}\sqrt{\sigma_1(\Sigma)}.
\end{equation}
Thus, if $\bar f_I(\tube_I)\le\frac{1}{2}\bar f_I$, replacing in \eqref{eq2:fi-fiwi} and using our assumption on $\bar f_I$, we get
\begin{equation}
\sigma_1(\Sigma)\ge\frac{1}{4\pi b}.  
\end{equation}
Let us now consider the case $\bar f_I(\tube_I)>\frac{1}{2}\bar f_I$. It implies that 
\[\sup_{s\in{[}0,|B_I|{]}}f(\tube_I,s)\ge \bar f_I(\tube_I)\ge \frac{1}{2\sqrt{2b|B_I|}}.\]
Since $\int_{\partial \Sigma} f=0$, we have $\sum_{j\neq I}\int_{B_j} f=\sum_{j\neq I}|B_j|\bar f_j=-|B_I|\bar f_I.$
Thus, there exists a geodesic boundary component $B_J$ such that 
\begin{equation}\label{eq2:fj}
|B_J|\bar f_J\le -\frac{|B_I|\bar f_I}{b-1}\le -\frac{\sqrt{|B_I|}}{(b-1)\sqrt{2b}}.
\end{equation}
Note that inequalities \eqref{eq2:l2normfi}--\eqref{eq2:fi-fiwi} hold for any $j$ and are not specific to $j=I$. In particular, we have
\begin{equation}\label{eq2:fjwj} 
\sqrt{|B_J|}| \bar f_J(\tube_J)-\bar f_J|\le\sqrt\frac{\pi\sigma_1(\Sigma)}{2}.\end{equation}
We also have
\begin{equation}
    \inf_{s\in{[}0,|B_J|{]}} f(\tube_J,s)\le {\bar f_J(\tube_J)}.
    \end{equation}
If $\inf_{s\in{[}0,|B_J|{]}} f(\tube_J,s)\ge\frac{1}{4\sqrt{2b}\sqrt{|B_I|}}$, then {$\bar f_J(\tube_J)\ge \frac{1}{4\sqrt{2b}\sqrt{|B_I|}}.$}
 It implies
{
\begin{align*}
\sqrt{|B_J|}( \bar f_J(\tube_J)-\bar f_J)&\ge\frac{\sqrt{|B_J|}}{4\sqrt{2b}\sqrt{|B_I|}}+  \frac{\sqrt{|B_I|}}{(b-1)\sqrt{2b}\sqrt{|B_J|}}\\
&\ge\frac{1}{4(b-1)\sqrt{2b}}\left(\frac{\sqrt{|B_J|}}{\sqrt{|B_I|}}+\frac{\sqrt{|B_I|}}{\sqrt{|B_J|}}\right)\\
&\geq \frac{1}{4(b-1)\sqrt{2b}}.
\end{align*}}
 Together with \eqref{eq2:fjwj}, we get
  { \[\sigma_1(\Sigma)\ge\frac{1}{16\pi b(b-1)^2}.\]}
 We now assume $\inf_{s\in{[}0,|B_J|{]}} f(\tube_J,s)<\frac{1}{4\sqrt{2b}\sqrt{|B_I|}}.$ Then
\begin{align}\label{eq2:lb}
   \sup_{s\in{[}0,|B_I|{]}}f(\tube_I,s)-\inf_{s\in{[}0,|B_J|{]}}f(\tube_J,s)&\ge \frac{1}{4\sqrt{2b}\sqrt{|B_I|}}\ge  \frac{1}{4\sqrt{2b\beta}}.
    \end{align}
    }

Let $p_I=(\tube_I,s_I)$ and $p_J=(\tube_J,s_J)$ (points are represented in the Fermi coordinates based on the corresponding geodesic $B_I$ and $B_J$) be such that 
\[f(p_I)=\sup_{s\in [0,|B_I|]} f(\tube_I,s),\quad\text{and}\quad f(p_J)=\inf_{s\in [0,|B_J|]} f(\tube_J,s).\]
Let us consider the $\eps$-thick-thin decomposition of $\Sigma$ as described in Section \ref{sec:per}. Note that $p_I, p_J\in \Sigma_{\thick}^\eps$. Let $$\curve:[0,1]\to \Sigma\setminus (\bigcup_{j=1}^b\collar_j^+\cup (\bigcup_{j=1}^p \cusp_j))$$ be an arbitrary curve connecting $p_I$ and $p_J$ with $\curve(0)=p_I$, $\curve(1)=p_J$. Moreover, we make the following additional assumptions.
\begin{itemize}
\item[a)]The interval
     $[0,1]$ admits  a partition $0=t_0<t_1<t_2<\cdots<t_{n-1}<t_n=1$ such that either $\curve([t_i,t_{i+1}])$  is a subset of a connected component of $\Sigma_{\thick}^\eps$, or  $\curve((t_i,t_{i+1}))\subset\collar(\gamma_j)^\circ\subset\Sigma_{\thin}^\eps$ for some $j$ with $\curve(t_{i+1})=\curve(t_i)^*\in \partial\collar(\gamma_j)$, where $\curve(t_i)^*$ is the reflection of $\curve(t_i)$ with respect to $\gamma_j$. 
     \item[b)] Each element of the collection $\{\curve((t_i,t_{i+1}))\}$ belongs to a separate connected components of $\Sigma_{\thick}^\eps$ or $\Sigma_{\thin}^\eps$.
     \end{itemize}

{The number of {connected components} of the thick and thin parts is bounded by  {$c_1|\chi|$, where $c_1$ is a positive universal constant.} If  {$\curve([t_i,t_{i+1}])$} is  in a connected component $\Sigma_i$ of $\Sigma_{\thick}^\eps$, then by Lemma~\ref{lem:sobinq} 
\begin{align*}
    |f\circ \curve(t_i)-f\circ \curve(t_{i+1})|\le {c_2 e^{\beta/2}\sqrt{|\Sigma_i|\sigma_1(\Sigma)}},
\end{align*}
{where $c_2$ is a positive universal constant such that the right-hand side is an upper bound for  $c(\beta)$ as given in Lemma \ref{lem:sobinq}. }
If there exists a curve $\curve$ as described above such that  {whenever $\curve([t_i,t_{i+1}])$ is a subset of the thin part, we have} 
$$|f\circ \curve(t_i)-f\circ \curve(t_{i+1})|\le  \frac{1}{8{c_1|\chi|}\sqrt{2b\beta}},$$ then 

\begin{align*} f(p_I)-f(p_J) &\le \sum_i|f\circ \curve(t_{i})-f\circ \curve(t_{i+1})| \\
&\le c_2 e^{\beta/2}\left(\sum_i \sqrt{|\Sigma_i|}\right)  \sqrt{\sigma_1(\Sigma)} + \frac{1}{8\sqrt{2b\beta}} \\
&\le c_2  e^{\beta/2}\sqrt{c_1 |\chi| |\Sigma|} \sqrt{\sigma_1(\Sigma)}  + \frac{1}{8 \sqrt{2b\beta}} \\
&\le c_3 e^{\beta/2} |\chi| \sqrt{\sigma_1(\Sigma)}  + \frac{1}{8 \sqrt{2b\beta}}.
\end{align*}

Combining it with \eqref{eq2:lb} we get

\[{\sigma_1(\Sigma)}\ge {\frac{1}{c_4b\beta e^\beta\chi^2}},\]

and we obtain the result. Here, $c_3$ and $c_4$ are positive universal constants.}

If such curve does not exist, it means that for any $\curve$ described above there  exists  an $i$ with $\curve|_{[t_i,t_{i+1}]}$  entirely in $\collar(\gamma_j)\subset\Sigma_{\thin}^\eps$ for some $j$ such that
\begin{equation}\label{eq2:thinpartub}
    |f\circ \curve(t_i)-f\circ \curve(t_{i+1})|> \frac{1}{8{c_1|\chi|}\sqrt{2b\beta}},
    \end{equation}
then by Lemma \ref{lem:tube-energy}, we have 
\begin{align*}
        \int_{\collar(\gamma_j)}|\nabla f|^2  \geq \frac{1}{2^7{c_1^2\chi^2}b\beta}|\gamma_j|.
    \end{align*}
     Let $\collar_{j_1},\cdots, \collar_{j_k}$ collection of such collars.   Hence, 
    \begin{align*}
    \sigma_1(\Sigma)=\int_\Sigma|\nabla f|^2&\ge \sum_m\int_{\collar_{j_m}}|\nabla f|^2\ge \frac{1}{2^7{c_1^2\chi^2}b\beta}\sum_m|\gamma_{j_m}|.
    \end{align*}
    But $\{\gamma_{j_m}\}$  must divide $\Sigma$ into at least two connected components one containing $B_I$ and the other $B_J$. Otherwise, $p_I$ and $p_J$ can be connected by a curve $\curve$ as described above such that there is no interval in the partition for which \eqref{eq2:thinpartub} holds. It contradicts our assumption. {Therefore}, it is clear that a subcollection of $\{\gamma_{j_m}\}$ gives us a multi-geodesic in $\collar_1$ and  we conclude that $\sum_m|\gamma_{j_m}|\ge \ell_1$. In summary, we obtain   
\begin{align} \label{eq:lowerbound} 
\sigma_1(\Sigma)&\ge c_5\,\min\left\{\frac{1}{b^3},\frac{1}{b\beta(1+\beta) e^{\beta}},\frac{1}{b\beta e^\beta\chi^2},\frac{\ell_1}{\chi^2b\beta}\right\}\\
\nonumber&{\ge \frac{c_5}{b\chi^2}\,\min\left\{\frac{1}{(1+\beta)^2 e^{\beta}},\frac{\ell_1}{\beta}\right\}}
\end{align}
where $c_5$ is a positive universal constant. We also observe that when $\beta$ is small enough, the minimum is achieved either by the first or the last term in the right-hand side of \eqref{eq:lowerbound}.

\end{proof}
{\begin{remark}\label{rem:ellepsilon}
 The proof of Proposition \ref{lem:ell_1} shows that $\ell_1$ appears in the lower bound of $\sigma_1$ only if there exists a multi-geodesic $c \in \chain_1$ such that the length of each closed geodesic in $c$ is at most $2\eps$. 
 {In particular, we can replace $\ell_1$ with  $\ell_1^\eps$ in Proposition \ref{lem:ell_1}. Here, $$\ell_k^\eps:=\inf\left\{|\curve|: \curve\in \chain_k\cap\Sigma_{\thin}^\eps\right\}.$$ We set $\ell_k^\eps=\infty$ if $\chain_k\cap\Sigma_{\thin}^\eps=\emptyset$. Note that, by abuse of notation,  $\curve\in\Sigma_{\thin}^\eps$  means that its image belongs to $\Sigma_{\thin}^\eps$.  When $\ell_k^\eps<\infty$, then $\ell_k^\eps=\ell_k$. It shows that when $\Sigma_{\thin}^\eps=\emptyset$, then the lower bound only depends on $\chi, b$ and $\beta$ and not on $\ell_1$. }
\end{remark}}

{The next theorem shows that the result of Proposition \ref{lem:ell_1} can be extended to some higher-order {Steklov} eigenvalues. }
\begin{theorem}\label{thm:swy2d}
    Let $\Sigma$ be a hyperbolic surface of signature $(g,b;p)$. Let
    \begin{align*}
        K = \begin{cases}
                b - 1 & \text{if $(g \geq 1 $ or $ p \geq 2)$ and $b \geq 1$}, \\
                b - 2 & \text{if $g = 0, p = 1$ and $b \geq 2$}, \\ 
                b - 3 & \text{if $g = 0, p = 0$ and $b \geq 3$}.
        \end{cases}
    \end{align*}
     
   Then there exists a positive universal constant $c$ such that
  
   \[
     \sigma_k(\Sigma)
 \ge \frac{c}{b|\chi|^3} \min\left\{\frac{1}{(1+\beta)^2e^\beta},\frac{\ellSt_k}{\beta}\right\},\qquad 0<k\le K,
    \]
    and
    \[
        \sigma_{K+1}(\Sigma) \geq \frac{c}{b\chi^2(1+\beta)^2e^\beta}. 
    \]
\end{theorem}

\begin{remark}
\label{remSN}
In order to prove this result, we will use a generalisation of Proposition \ref{lem:ell_1} to the first non-zero mixed Steklov-Neumann eigenvalue of a hyperbolic surface $\Sigma$ {of signature $(g,b^*;p)$} with Steklov condition on {$b<b^*$} geodesic boundary components $\{B_1,\dots,B_b\}$ and Neumann condition on the remaining boundary components $\{B_{b+1},\dots, B_{b^*}\}$. The result involves the quantity $\ellSt_1^* := \inf_{\curve \in {\chain^*_1}} |\curve|,$
    where ${\chain_1^*}$ denotes the set of multi-geodesics formed by disjoint simple closed geodesics, dividing $\Sigma$ into two connected components, each containing at least one boundary component with Steklov condition. When ${\chain_1^*=\emptyset}$, we set $\ell_1^*=\infty$. 
    Let $\beta^*:=\max\{|B_i|, i=1,\dots, b^*\}$. {Then $\sigma_1^N(\Sigma)$ has the same lower bound as given in Proposition~\ref{lem:ell_1} with $\ell_1$ 
    replaced with $\ell_1^*$ 
    and $\beta$ with $\beta^*$.
    \begin{equation}\label{eq:lowerbound for Steklov-Neumann}
          \sigma_1^N(\Sigma) \geq
           \frac{c}{b\chi^2} \min\left\{\frac{1}{(1+\beta^*)^2e^{\beta^*}},\frac{\ellSt^*_1}{\beta^*}\right\}. 
           \end{equation}
           }
    The proof is exactly the same as the proof of Proposition \ref{lem:ell_1}. 
    If the maximum length of the boundary components with Neumann condition is smaller than {$\max\{\beta,2\arsinh(1)\}$, where $\beta=\max\{|B_i|: i=1,\ldots, b\}$, then we can replace $\beta^*$ with $\beta$, and $\ell_1^*$ with $\ell_1^{*,\eps}$ in \eqref{eq:lowerbound for Steklov-Neumann}, taking into account Remark \ref{rem:ellepsilon}.} 
\end{remark}
{Theorem \ref{thm:swy2d} holds if we replace $\ell_k$ with $\ell_k^\eps$ and the statements are equivalent.} Hence, we prove it in this case using Remarks~\ref{rem:ellepsilon} and \ref{remSN} to simplify the argument.
\begin{proof}[Proof of Theorem \ref{thm:swy2d}] For a given $k\in \{1,\ldots,K+1\}$, let $1\leq s\leq k$ be the largest $s$ such that  $\ell_s^\eps\neq\infty$. If such $s$ does not exist then $\ell_1^\eps=\infty$ and the result immediately follows from Proposition~\ref{lem:ell_1} together with Remark \ref{rem:ellepsilon}.

Let us first consider the case when $s=k$. Note that it automatically implies that $k\le K$. We consider a curve $\curve=\gamma_1 \cup \dots \cup \gamma_p{\in \chain_k\cap\Sigma_{\thin}^\eps}$ with $|\curve|={\ellSt_k^\eps}$. Since $|\curve|= {\ellSt_k^\eps}$, one of the $p$ components of $\curve$ must be of length $\geq \frac{{\ellSt_k^\eps}}{p}$; we call it $\gamma_{\max}$. We decompose $\Sigma$ into $k$ components $\Sigma_1,\ldots,\Sigma_k$ containing at least one boundary component by removing from $\Sigma$ all the geodesics of $\curve$ except $\gamma_{\max}$.  On each $\Sigma_i$, we consider the mixed Steklov-Neumann problem with Steklov condition on $\Sigma_i\cap \partial \Sigma$ and Neumann condition on $\partial \Sigma_i{\cap\Sigma}$. Since the {$\Sigma_i$'s} are disjoint, by standard variational argument, we have
 \[
 \sigma_k(\Sigma)\ge\min\{\sigma_1^N(\Sigma_1),\ldots,\sigma_1^N(\Sigma_{k})\}
 \]
Because {$\curve\in \chain_k\cap \Sigma_{\thin}^\eps$}, all the boundaries of $\Sigma_i$ with Neumann condition are of length $\leq {2\eps}\le 2\arsinh(1)$. {Hence,} we have from Remark \ref{remSN} that
$\sigma_1^N(\Sigma_i)\ge \frac{c_6}{b\chi^2}\,\min\left\{\frac{1}{(1+\beta)^2 e^{\beta}},\frac{{\ell_1^{*,\eps}(\Sigma_i)}}{\beta}\right\}$. 
{Moreover,}
 $\ellSt_1^{*,\eps}(\Sigma_i)\geq |\gamma_{\max}|\geq \frac{{\ellSt_k^\eps}}{p}$,  because {otherwise} it would contradict the fact that $c$ is minimal. We also have $p\leq c_7|\chi|$, for some positive universal constant $c_7$. Therefore, $$\sigma_k(\Sigma)\geq \frac{c_8}{b|\chi|^3}\,\min\left\{\frac{{\chi}}{(1+\beta)^2 e^{\beta}},\frac{{\ell_k^\eps}}{\beta}\right\}.$$
 
{We now consider the case when $1\le s<k$ and} we show that $\sigma_k(\Sigma)\geq \frac{c_6}{b\chi^2(1+\beta)^2 e^{\beta}}$. We consider a curve {$\curve_{s}\in \chain_s\cap\Sigma_{\thin}^\eps$}  such that ${|\curve_{s}|=\ellSt_{s}^\eps}$. We decompose $\Sigma$ into {$s+1$} components $\Sigma_1,\ldots,\Sigma_{{s+1}}$ containing at least one boundary component by removing from $\Sigma$ all the geodesics of $\curve_s$. On each $\Sigma_i$, we consider the mixed Steklov-Neumann problem with Steklov condition on $\Sigma_i\cap \partial \Sigma$ and Neumann condition on $\partial \Sigma_i{\cap\Sigma}$. We have
 \[
{\sigma_k(\Sigma)\ge \sigma_{s+1}}(\Sigma)\geq\min\{\sigma_1^N(\Sigma_1),\ldots,\sigma_1^N(\Sigma_{s+1})\}.
 \]

{Again,} because all the boundaries of $\Sigma_i$ with Neumann condition are of length $\le {2\eps}\le 2\arsinh(1)$, we have from Remark \ref{remSN} that 
$\sigma_1^N(\Sigma_i)\ge \frac{c_6}{b\chi^2}\,\min\left\{\frac{1}{(1+\beta)^2 e^{\beta}},\frac{{\ell_1^{*,\eps}}(\Sigma_i)}{\beta}\right\}$. We claim that $\ell_1^{*,\eps}(\Sigma_i)=\infty$ for all $i$. If there exists $I$ for which there exist $\curve_I\in\chain^*_1(\Sigma_I)\cap\Sigma_{\thin}^\eps$, then $\curve_I\cup \curve\in \chain_{s+1}\cap \Sigma_{\thin}^\eps$ and it contradicts the maximality of $s$.   
In summary, we obtain
\[
\sigma_k(\Sigma)\geq \frac{c_8}{b|\chi|^3}\,\min\left\{\frac{{\chi}}{(1+\beta)^2 e^{\beta}},\frac{{\ellSt_k^\eps}}{\beta}\right\}, \qquad 1\le k\le K,
\]
and
\[
\sigma_{K+1}\geq  \frac{c_6}{b\chi^2(1+\beta)^2 e^{\beta}}.
\]
 
\end{proof}
We now give a remark on {upper bounds}.
\begin{remark}\label{rem:upperbound}
For every $k\ge1$, the bounds of the form 
$\sigma_k(\Sigma)|\partial \Sigma|\le c(|\chi|+k)$
is known for compact surfaces {\cite{Kar,H11} (see also \cite{GP12,CGE})} and it remains true in the setting of finite volume surfaces. As a result, we get
\[\sigma_k\le c_1\frac{|\chi|}{\beta}.\]
Using a classical comparison argument with the Steklov-Dirichlet eigenvalue on half-collar near the boundary, we get  (see \cite[Lemma 3]{Per22}):
$$\sigma_k 
\le c_2e^\beta,\qquad 1\le k<b.$$
{It has been shown in \cite{Per22} that if $\ellSt_k$ is sufficiently small, $\sigma_k\leq c_3\frac{\ellSt_k}{\alpha}$, {where $\alpha$ is the minimum length of boundary components}. The proof is by constructing appropriate test functions around the collar or half collars and constant elsewhere. More precisely,} by \cite[Proof of Theorem 3]{Per22}, we have the following upper bound for $1\le k\le K$. Let $\curve\in \chain_k$ such that $|\curve|=\ell_k$. Let $\Sigma_j$ be the connected components of $\Sigma\setminus\curve$, and let $L_j$ denote the length of ${\Sigma_j}\cap\partial\Sigma$. Then
\[\sigma_k\le \max_j\frac{\ell_k}{L_j\arctan\left(\frac{1}{\sinh\left(\frac{\ell_k}{2}\right)}\right)}\le \max_j\frac{1}{L_j}\ell_ke^{\ell_k/2}\le \frac{\ell_ke^{\ell_k/2}}{\alpha},\qquad 1\le k\le K.\]
 Note that if there exists $\curve\in\chain_k$ with $|\curve|=\ell_k$ such that $L_j$ are comparable to $\beta$ then we can replace $\alpha$ by $\beta$. 
 
In summary, we have
 $$\sigma_k\le c_4\min\left\{{e^{\beta}},{\frac{|\chi|}{\beta}},\frac{\ell_ke^{\ell_k/2}}{\alpha}\right\}, 
 \quad 1\le k<K,\quad
 \text{and}\quad \sigma_{K+1}
 \le c_1{\frac{|\chi|}{\beta}}.$$
\end{remark}

{All the results above can be extended if the curvature assumption is relaxed to a pinched negatively curved surface. We first consider the compact setting. 
\begin{theorem}\label{rem:pinched curvature2}   
 Let $(\Sigma,h)$ be a compact Riemannian surface with geodesic boundary, and assume its Gaussian curvature lies in the interval {$[-1, -\kappa]$}, where $\kappa \in (0,1)$. Then, the conclusions of Proposition~\ref{lem:ell_1}, Theorem~\ref{thm:swy2d}, 
 and 
{Remark \ref{rem:upperbound}}
 still hold, with $\kappa$ {and $\kappa^{-1/2}$ respectively appearing} as a multiplicative factor in the lower  {and upper} bounds.
 \end{theorem}
 }
 \begin{proof}
    Let $\double\Sigma$ denote the double of $\Sigma$. Note that $\double\Sigma$ is a closed surface equipped with an involution whose fixed points are precisely the geodesic boundary. If we use the normalised Ricci flow to deform its metric conformally to the unique hyperbolic metric in the same conformal class (see \cite{CK04}), then, because the normalised Ricci flow preserves isometries, the resulting hyperbolic metric $\bar{h}$ also admits the same involution symmetry along the boundary points. Hence, each component of $\partial \Sigma$  is geodesic in $(\double\Sigma, \bar h)$ because they are fixed points of an isometry.   Therefore, we can view $(\Sigma, \bar{h})$ as a subdomain of $(\double\Sigma, \bar{h})$, whose boundary is geodesic.

Next, write $h = \rho \bar{h}$. Similar to what is done in \cite{SWY80},  by a generalisation of the Ahlfors--Schwarz argument (see \cite{Wol}), we have $
1 \leq \rho \leq \frac{1}{\kappa},$
and therefore, 
\[
\sqrt{\kappa} \,\sigma_{k}(\Sigma, \bar{h})\le \sigma_k(\Sigma, h) \leq \sigma_{k}(\Sigma, \bar{h}),\quad \text{and}\quad \ell_k(\Sigma,\bar h)\le \ell_k((\Sigma,h)\le\frac{1}{\sqrt{\kappa}}\ell_k(\Sigma,\bar h).
\]
The statement then follows. 

 \end{proof}

 We end with a remark on the non-compact finite-volume setting. 
 \begin{remark}
     In the statement of Theorem \ref{rem:pinched curvature2}, if we remove the compactness assumption and assume that the surface has finite volume, then similar results hold.

   Indeed, if $h$ is asymptotic to a multiple of a hyperbolic cusp metric on each end of $\Sigma$, we can adapt the above argument and perform the normalised Ricci flow ~\cite{LMS09} {and obtain the same statement as in the compact case.} 
If this is  not case, we can adapt the proofs of Proposition \ref{lem:ell_1} and Theorem~\ref{thm:swy2d} to get similar results with bounds depending also on $\kappa$ and the expression in terms of $\beta$ may be slightly different. {We give hereafter an idea of how this is done.} 
The $\eps$-thick-thin decomposition defined in Section \ref{sec:per} holds when the Gaussian curvature of a Riemannian surface $(\Sigma,h)$ with geodesic boundary lies in the interval $[-1, -\kappa]$, where $\kappa \in (0,1)$. The primary difference is that the tubes and cusps are diffeomorphic, rather than isometric, to the warped product sets described above. However, the definition of $\eps$ remains unchanged and independent of $\kappa$ (see \cite[Theorem 4.3.2]{Bus92} and \cite[\S 10]{BGS85}).

Furthermore, Lemma \ref{lem:tube-energy} can be established by comparing the Dirichlet energy of a function on a collar around a simple closed geodesic with that in hyperbolic space. Specifically, on $\collar(\gamma)$, the metric $h$ can be expressed in Fermi coordinates along $\gamma$ as $h = \dd t^2 + A(t,s)^2 \dd s^2$. By writing the Gaussian curvature in terms of $A$ and using the curvature bounds, we obtain two second-order differential inequalities which imply  $$\cosh(\sqrt{\kappa}t)\le|A(t,s)|\le \cosh(t),$$ which allows one to adapt the proof of Dodziuk and Randol \cite[Lemma 3]{DR86}. 

{Lemma \ref{lem:sobinq} also holds in this context; for a proof, refer to \cite[Lemma 4.6]{BBHM}. The constant coefficient in this lemma depends on the norm of the derivatives of the Gaussian curvature. However, there exists a metric quasi-isometric to the original one with some nice properties, where the norm of the derivative of the curvature is bounded (see \cite[Remark 4.7]{BBHM}). Therefore, only in the final step of the proof for Case 2 in Proposition 3.3 we need to switch to the quasi-isometric metric before applying Lemma 2.2. A similar approach is used in \cite{BBHM}.
   }

 \end{remark}
\noindent

\bibliographystyle{alpha}
\bibliography{HyperbolicRef}

\begin{thebibliography}{CESG11}

\bibitem[BBHM]{BBHM}
Ara Basmajian, Jade Brisson, Asma Hassannezhad, and Antoine Métras.
\newblock Tubes and steklov eigenvalues in negatively curved manifolds.
\newblock Int. Math. Res. Not. IMRN, to apper.

\bibitem[BGS85]{BGS85}
Werner Ballmann, Mikhael Gromov, and Viktor Schroeder.
\newblock {\em Manifolds of nonpositive curvature}, volume~61 of {\em Progress in Mathematics}.
\newblock Birkh\"{a}user Boston, Inc., Boston, MA, 1985.

\bibitem[BPS12]{BPS}
Florent Balacheff, Hugo Parlier, and St\'ephane Sabourau.
\newblock Short loop decompositions of surfaces and the geometry of {J}acobians.
\newblock {\em Geom. Funct. Anal.}, 22(1):37--73, 2012.

\bibitem[Bre02a]{Bre02a}
Simon Brendle.
\newblock Curvature flows on surfaces with boundary.
\newblock {\em Math. Ann.}, 324(3):491--519, 2002.

\bibitem[Bre02b]{Bre02b}
Simon Brendle.
\newblock A family of curvature flows on surfaces with boundary.
\newblock {\em Math. Z.}, 241(4):829--869, 2002.

\bibitem[Bur88]{burger88}
Marc Burger.
\newblock Asymptotics of small eigenvalues of {R}iemann surfaces.
\newblock {\em Bull. Amer. Math. Soc. (N.S.)}, 18(1):39--40, 1988.

\bibitem[Bur90]{burger90}
Marc Burger.
\newblock Small eigenvalues of {R}iemann surfaces and graphs.
\newblock {\em Math. Z.}, 205(3):395--420, 1990.

\bibitem[Bus92]{Bus92}
Peter Buser.
\newblock {\em Geometry and spectra of compact {R}iemann surfaces}, volume 106 of {\em Progress in Mathematics}.
\newblock Birkh\"{a}user Boston, Inc., Boston, MA, 1992.

\bibitem[CESG11]{CGE}
Bruno Colbois, Ahmad El~Soufi, and Alexandre Girouard.
\newblock Isoperimetric control of the {S}teklov spectrum.
\newblock {\em J. Funct. Anal.}, 261(5):1384--1399, 2011.

\bibitem[Cha21]{Cha}
Asad Chaudhary.
\newblock {\em Estimates for Small Eigenvalues of the Laplacian and Conformal Laplacian on Closed Manifolds}.
\newblock PhD thesis, University of Oxford, 2021.

\bibitem[CK04]{CK04}
Bennett Chow and Dan Knopf.
\newblock {\em The {R}icci flow: an introduction}, volume 110 of {\em Mathematical Surveys and Monographs}.
\newblock American Mathematical Society, Providence, RI, 2004.

\bibitem[Col85]{Col85}
Bruno Colbois.
\newblock Petites valeurs propres du laplacien sur une surface de riemann compacte et graphes.
\newblock {\em C. R. Acad. Sci. Paris Sér. I Math}, 20(301):927--930, 1985.

\bibitem[Dod87]{Dod}
Jozef Dodziuk.
\newblock A lower bound for the first eigenvalue of a finite-volume negatively curved manifold.
\newblock {\em Bol. Soc. Brasil. Mat.}, 18(2):23--34, 1987.

\bibitem[DPRS87]{DRS87}
Jozef Dodziuk, Thea Pignataro, Burton Randol, and Dennis Sullivan.
\newblock Estimating small eigenvalues of riemann surfaces.
\newblock {\em The legacy of Sonya Kovalevskaya (Cambridge, Mass., and Amherst, Mass., 1985), Contemp. Math}, 64:93--121, 1987.

\bibitem[DR86]{DR86}
Jozef Dodziuk and Burton Randol.
\newblock Lower bounds for {$\lambda_1$} on a finite-volume hyperbolic manifold.
\newblock {\em J. Differential Geom.}, 24(1):133--139, 1986.

\bibitem[Esc97]{Esc97}
Jos\'e~F. Escobar.
\newblock The geometry of the first non-zero {S}tekloff eigenvalue.
\newblock {\em J. Funct. Anal.}, 150(2):544--556, 1997.

\bibitem[Esc99]{Esc99}
Jos\'e~F. Escobar.
\newblock An isoperimetric inequality and the first {S}teklov eigenvalue.
\newblock {\em J. Funct. Anal.}, 165(1):101--116, 1999.

\bibitem[GP10]{GP10}
Alexandre Girouard and Iosif Polterovich.
\newblock On the {H}ersch-{P}ayne-{S}chiffer estimates for the eigenvalues of the {S}teklov problem.
\newblock {\em Funktsional. Anal. i Prilozhen.}, 44(2):33--47, 2010.

\bibitem[GP12]{GP12}
Alexandre Girouard and Iosif Polterovich.
\newblock Upper bounds for {S}teklov eigenvalues on surfaces.
\newblock {\em Electron. Res. Announc. Math. Sci.}, 19:77--85, 2012.

\bibitem[GR19]{GR19}
Nadine Gro{\ss}e and Melanie Rupflin.
\newblock Sharp eigenvalue estimates on degenerating surfaces.
\newblock {\em Comm. Partial Differential Equations}, 44(7):573--612, 2019.

\bibitem[Has11]{H11}
Asma Hassannezhad.
\newblock Conformal upper bounds for the eigenvalues of the {L}aplacian and {S}teklov problem.
\newblock {\em J. Funct. Anal.}, 261(12):3419--3436, 2011.

\bibitem[HHH24]{HHH22}
Xiaolong~Hans Han, Yuxin He, and Han Hong.
\newblock Large {S}teklov eigenvalues on hyperbolic surfaces.
\newblock {\em Math. Z.}, 308(2):Paper No. 37, 21, 2024.

\bibitem[HM20]{HM}
Asma Hassannezhad and Laurent Miclo.
\newblock Higher order {C}heeger inequalities for {S}teklov eigenvalues.
\newblock {\em Ann. Sci. \'Ec. Norm. Sup\'er. (4)}, 53(1):43--88, 2020.

\bibitem[Jam15]{Jam15}
Pierre Jammes.
\newblock Une in\'egalit\'e de {C}heeger pour le spectre de {S}teklov.
\newblock {\em Ann. Inst. Fourier (Grenoble)}, 65(3):1381--1385, 2015.

\bibitem[JMS09]{LMS09}
Lizhen Ji, Rafe Mazzeo, and Natasa Sesum.
\newblock Ricci flow on surfaces with cusps.
\newblock {\em Math. Ann.}, 345(4):819--834, 2009.

\bibitem[Kar17]{Kar}
Mikhail Karpukhin.
\newblock Bounds between {L}aplace and {S}teklov eigenvalues on nonnegatively curved manifolds.
\newblock {\em Electron. Res. Announc. Math. Sci.}, 24:100--109, 2017.

\bibitem[LMP23]{LMP}
Michael Levitin, Dan Mangoubi, and Iosif Polterovich.
\newblock {\em Topics in spectral geometry}, volume 237 of {\em Graduate Studies in Mathematics}.
\newblock American Mathematical Society, Providence, RI, [2023] \copyright2023.

\bibitem[M{\"{u}}l92]{Mul92}
Werner M{\"{u}}ller.
\newblock Spectral geometry and scattering theory for certain complete surfaces of finite volume.
\newblock {\em Invent. Math.}, 109(2):265--305, 1992.

\bibitem[OPS88]{OPS88}
Brad Osgood, Ralph Phillips, and Peter Sarnak.
\newblock Extremals of determinants of {L}aplacians.
\newblock {\em J. Funct. Anal.}, 80(1):148--211, 1988.

\bibitem[Per23]{Per23}
H\'el\`ene Perrin.
\newblock {\em In\'egalit\'e g\'eom\'etriques pour des valeurs propres de Steklov de graphes et de surfaces}.
\newblock PhD thesis, Universit\'e de Neuch\^atel, Neuch\^atel, Suisse, 2023.

\bibitem[Per24]{Per22}
H\'el\`ene Perrin.
\newblock Estimates for low steklov eigenvalues of surfaces with several boundary components.
\newblock {\em Annales math{\'e}matiques du Qu{\'e}bec}, 2024.
\newblock DOI 10.1007/s40316-024-00221-y.

\bibitem[Pol23]{Pol21}
Panagiotis Polymerakis.
\newblock On the {S}teklov spectrum of covering spaces and total spaces.
\newblock {\em Ann. Global Anal. Geom.}, 63(1):Paper No. 10, 22, 2023.

\bibitem[Rup21]{Rup21}
Melanie Rupflin.
\newblock Hyperbolic metrics on surfaces with boundary.
\newblock {\em J. Geom. Anal.}, 31(3):3117--3136, 2021.

\bibitem[SWY80]{SWY80}
Richard Schoen, Scott Wolpert, and Shing-Tung Yau.
\newblock Geometric bounds on the low eigenvalues of a compact surface.
\newblock In {\em Geometry of the {L}aplace operator ({P}roc. {S}ympos. {P}ure {M}ath., {U}niv. {H}awaii, {H}onolulu, {H}awaii, 1979)}, Proc. Sympos. Pure Math., XXXVI, pages 279--285. Amer. Math. Soc., Providence, R.I., 1980.

\bibitem[Wol82]{Wol}
Scott Wolpert.
\newblock A generalization of the {A}hlfors-{S}chwarz lemma.
\newblock {\em Proc. Amer. Math. Soc.}, 84(3):377--378, 1982.

\bibitem[WX22]{YY}
Yunhui Wu and Yuhao Xue.
\newblock Optimal lower bounds for first eigenvalues of {R}iemann surfaces for large genus.
\newblock {\em Amer. J. Math.}, 144(4):1087--1114, 2022.

\end{thebibliography}
\end{document}